\documentclass[10pt, reqno]{amsart}
\usepackage{amsmath, amsthm, amscd, amsfonts, amssymb, graphicx, color}
\usepackage[bookmarksnumbered, colorlinks, plainpages]{hyperref}
\hypersetup{colorlinks=true,linkcolor=red, anchorcolor=green, citecolor=cyan, urlcolor=red, filecolor=magenta, pdftoolbar=true}
\usepackage{mathrsfs}

\newtheorem{theorem}{Theorem}[section]
\newtheorem{lemma}[theorem]{Lemma}

\newtheorem{corollary}[theorem]{Corollary}
\theoremstyle{definition}
\newtheorem{definition}[theorem]{Definition}

\theoremstyle{remark}
\newtheorem{remark}[theorem]{Remark}
\numberwithin{equation}{section}

\begin{document}
\setcounter{page}{1}

\title[Galois cohomology revisited]{Galois cohomology revisited}

\author[Nikolaev]
{Igor V. ~Nikolaev$^1$}

\address{$^{1}$ Department of Mathematics and Computer Science, St.~John's University, 8000 Utopia Parkway,  
New York,  NY 11439, United States.}
\email{\textcolor[rgb]{0.00,0.00,0.84}{igor.v.nikolaev@gmail.com}}





\begin{abstract}
We recast the  Galois cohomology  of the  variety $V$  over a number field $k$ 
in terms of the  K-theory of a $C^*$-algebra  $\mathscr{A}_V$ connected to  $V$. 
 It is proved that $V$ is isomorphic to $V'$ over $k$ (algebraic closure of $k$, resp.) 
 if and only if  $\mathscr{A}_V$ is isomorphic  (Morita equivalent, resp.) to 
 $\mathscr{A}_{V'}$. 
 In particular, the Morita equivalent $C^*$-algebras  $\mathscr{A}_V$ parametrize 
 twists of the variety $V$. 
 The case of rational elliptic curves is considered in detail. 
  \end{abstract}

\maketitle

\section{Introduction}
Let $V$ be a complex projective variety given by a  homogeneous coordinate ring $\mathscr{A}$. 
If $k\subset\mathbf{C}$ is a subfield of complex numbers and $V(k)$ is the set of  $k$-points of $V$, 
then the isomorphisms of $V(k)$ over $\mathbf{C}$ cannot be defined over the field $k$
in general.   In such a case  the variety $V'(k)$ is called a  twist of $V(k)$. 
Equivalently,   $V'(k)$ is a twist,  if it is isomorphic to  $V(k)$  over  the field $\mathbf{C}$ but not 
over the smaller field $k$.   The twists of $V(k)$ are  classified  in terms of  the Galois cohomology   
 [Serre 1997]  \cite[p. 123]{S}.

 Recall that the Serre $C^*$-algebra  $\mathscr{A}_V$
 is defined as the norm closure of a self-adjoint representation of the twisted 
 homogeneous coordinate ring of  $V(k)$ by the bounded linear operators on a Hilbert space $\mathcal{H}$,
 see [Stafford \& van ~den ~Bergh 2001] \cite{StaVdb1} and \cite[Section 5.3.1]{N}. 
 The separable $C^*$-algebra $A$ is said to be   Morita equivalent (stably isomorphic) 
  to $A'$, if $A\otimes\mathscr{K}\cong A'\otimes\mathscr{K}$, where $\mathscr{K}$ is the $C^*$-algebra
 of all compact operators  on $\mathcal{H}$ and $\cong$ is an isomorphism of the $C^*$-algebras
  [Blackadar 1986]  \cite[Section 13.7.1]{B}.
  (Notice that if $A\cong A'$ are isomorphic $C^*$-algebras, then $A$ is  Morita equivalent to $A'$.)
  The correspondence $V(k)\mapsto\mathscr{A}_V$ is  a functor,   which maps 
 $\mathbf{C}$-isomorphic varieties to the Morita equivalent 
  Serre $C^*$-algebras \cite[Theorem 5.3.3]{N}.

\medskip
The aim of our note is a classification of the twists  of  $V(k)$ in terms of   the Serre $C^*$-algebra
$\mathscr{A}_V$.   Namely, we prove that if the variety $V(k)$ is $k$-isomorphic to a variety  $V'(k)$, 
 then  the  Serre $C^*$-algebra $\mathscr{A}_V$ is isomorphic to 
  $\mathscr{A}_{V'}$,  while   if $V(k)$ is $\mathbf{C}$-isomorphic to $V'(k)$,
then the Serre $C^*$-algebra $\mathscr{A}_V$ is Morita equivalent to $\mathscr{A}_{V'}$,
see corollary \ref{cor1.2}.   To formalize  our results, we need  the following definitions.

\medskip
Let $A$ be  a unital $C^*$-algebra and denote by  $V(A)$ the union  of projections in all the $n\times n$
matrix $C^*$-algebra with entries in $A$  [Blackadar 1986]  \cite[Section 5]{B}.  Recall that projections 
$p,q\in V(A)$ are called equivalent,  if there exists a partial
isometry $u$ such that $p=u^*u$ and $q=uu^*$. The corresponding equivalence
class is denoted by $[p]$.   The equivalence classes of  projections can be made into
a semigroup  with addition defined by the formula $[p]+[q]=[p+q]$. The Grothendieck
completion of the semigroup to an abelian group, $K_0(A)$,  is called the  $K_0$-group  of the 
algebra $A$.
The functor $A\to K_0(A)$ maps the category of unital
$C^*$-algebras into the category of abelian groups, so that
projections in the algebra $\cup_{n=1}^{\infty} M_n(A)$ correspond to a positive
cone  $K_0^+(A)\subset K_0(A)$ and the unit element $1\in A$
corresponds to an order unit $u\in K_0^+(A)$.
Scaled ordered group   $(K_0(A),K_0^+(A),u)$  with a fixed
order unit $u$  tracks projections in the original algebra $A$,
while $(K_0(A), K_0^+(A))$ is known as  a dimension group
 [Blackadar 1986]  \cite[Section 6]{B}.

An  AF-algebra  $\mathscr{B}$ (Approximately Finite $C^*$-algebra) is 
 the  norm closure of an ascending sequence of  the direct sum of  finite-dimensional
$C^*$-algebras $M_n(\mathbf{C})$,  where  $M_n(\mathbf{C})$ is the $C^*$-algebra of the $n\times n$ matrices
with entries in $\mathbf{C}$ [Blackadar 1986]  \cite[Section 7.1]{B}.
The  scaled dimension group $(K_0(\mathscr{B}), K_0^+(\mathscr{B}), u)$ is an isomorphism
invariant of the algebra $\mathscr{B}$.  In contrast, the dimension group
$(K_0(\mathscr{B}), K_0^+(\mathscr{B}))$ is an invariant of the  Morita equivalence of the AF-algebra 
$\mathscr{B}$ [Blackadar 1986]  \cite[Section 7.3]{B}.
The Serre $C^*$-algebra $\mathscr{A}_V$ is not an AF-algebra, but 
there exists a dense embedding $\mathscr{A}_V\hookrightarrow\mathscr{B}$,
where $\mathscr{B}$ is an AF-algebra  such that 
$(K_0(\mathscr{A}_V), K_0^+(\mathscr{A}_V))\cong (K_0(\mathscr{B}), K_0^+(\mathscr{B}))$
\cite[Lemma 3.1]{Nik1}.

Let $H^1(Gal(\bar k|k), Aut_{\mathbf{C}}^{~ab}(V))$ be  the first  Galois cohomology group  
of the extension $k\subset\mathbf{C}$,  where  $Aut_{\mathbf{C}}^{~ab}(V)$ is the maximal abelian subgroup of the group of 
$\mathbf{C}$-automorphisms  of the variety  $V(k)$ [Serre 1997]  \cite[p. 123]{S}.   
 Our main results can be formulated as follows. 
\begin{theorem}\label{thm1.1}
The group $H^1(Gal(\bar k|k), Aut_{\mathbf{C}}^{~ab}(V))$ has  structure of  
a dimension group of stationary type,  which is order-isomorphic to the group
$(K_0(\mathscr{A}_V), K_0^+(\mathscr{A}_V))$. 
\end{theorem}
\begin{corollary}\label{cor1.2}
Let $\mathscr{A}_V=F(V(k))$ and  $\mathscr{A}_{V'}=F(V'(k))$,
where $V(k)$ and $V'(k)$ are complex projective varieties over 
the field $k\subset\mathbf{C}$.  Then: 

\medskip
(i) $V(k)$ and $V'(k)$ are isomorphic over $k$  if and only if the Serre $C^*$-algebras
$\mathscr{A}_V\cong\mathscr{A}_{V'}$ are isomorphic; 

\smallskip
(ii)  $V(k)$ and $V'(k)$ are isomorphic over $\mathbf{C}$  if and only if the Serre  $C^*$-algebras
$\mathscr{A}_V$ and $\mathscr{A}_{V'}$ are Morita equivalent. 
\end{corollary}
The article is organized as follows. In Section 2 we briefly review the Serre $C^*$-algebras and the Galois cohomology.  
Theorem \ref{thm1.1} and corollary \ref{cor1.2} are 
proved  in Section 3.  An illustration of corollary \ref{cor1.2} can be found in Section 4.

\section{Preliminaries}
In this section we briefly review  Galois cohomology and Serre 
$C^*$-algebras. We refer the reader to  [Serre 1997]  \cite[Chapter I, \S 5]{S} and \cite[Section 5.3.1]{N}
for a detailed account. 

\subsection{Serre $C^*$-algebras}
Let $V$ be an $n$-dimensional complex  projective variety endowed with an automorphism $\sigma:V\to V$  
 and denote by $B(V, \mathcal{L}, \sigma)$   its  twisted homogeneous coordinate ring  [Stafford \& van ~den ~Bergh 2001]  \cite{StaVdb1}.
Let $R$ be a commutative  graded ring,  such that $V=Proj~(R)$.  Denote by $R[t,t^{-1}; \sigma]$
the ring of skew Laurent polynomials defined by the commutation relation
$b^{\sigma}t=tb$  for all $b\in R$, where $b^{\sigma}$ is the image of $b$ under automorphism 
$\sigma$.  It is known, that $R[t,t^{-1}; \sigma]\cong B(V, \mathcal{L}, \sigma)$.

Let $\mathcal{H}$ be a Hilbert space and   $\mathcal{B}(\mathcal{H})$ the algebra of 
all  bounded linear  operators on  $\mathcal{H}$.
For a  ring of skew Laurent polynomials $R[t, t^{-1};  \sigma]$,  
 consider a ring homomorphism: 
\begin{equation}\label{eq2.1}
\rho: R[t, t^{-1};  \sigma]\longrightarrow \mathcal{B}(\mathcal{H}). 
\end{equation}
Recall  that  $\mathcal{B}(\mathcal{H})$ is endowed  with a $\ast$-involution;
the involution comes from the scalar product on the Hilbert space $\mathcal{H}$. 
We shall call representation (\ref{eq2.1})  $\ast$-coherent,   if
(i)  $\rho(t)$ and $\rho(t^{-1})$ are unitary operators,  such that
$\rho^*(t)=\rho(t^{-1})$ and 
(ii) for all $b\in R$ it holds $(\rho^*(b))^{\sigma(\rho)}=\rho^*(b^{\sigma})$, 
where $\sigma(\rho)$ is an automorphism of  $\rho(R)$  induced by $\sigma$. 
(Notice that the second condition says that $\sigma$  preserves $\ast$-involution.
Moreover, $\sigma(\rho)$ is implemented by a unitary operator on $\mathcal{H}$.)   
Whenever  $B=R[t, t^{-1};  \sigma]$  admits a $\ast$-coherent representation,
$\rho(B)$ is a $\ast$-algebra.  
\begin{definition}
 The  Serre $C^*$-algebra $\mathscr{A}_V$  is  the  norm closure of  an $\ast$-coherent  representation 
 $\rho(B)$. 
 \end{definition}
\begin{remark}
The $\ast$-coherent representation $\rho(B)$ gives rise to the
unique (canonical)  automorphism $\sigma$.  
Likewise, the Serre $C^*$-algebra is independent of the choice 
of $\ast$-coherent representation of  $\rho(B)$. 
Thus the $C^*$-algebra $\mathscr{A}_V$ is correctly defined. 
\end{remark}

\subsection{Galois cohomology}
Let $G$ be a group. The set $\mathbf{A}$ is called a {\it $G$-set}, if 
$G$ acts on $\mathbf{A}$ on the left continuously. 
If $\mathbf{A}$ is a group and $G$ acts on $\mathbf{A}$ by the group
homomorphisms, then $\mathbf{A}$ is called a {\it $G$-group}. In particular,
if $\mathbf{A}$ is abelian, one gets a {\it $\mathbf{G}$-module.} 

If $\mathbf{A}$ is a $G$-group,  then a {\it 1-cocycle} of $G$ in $\mathbf{A}$ 
is a map $s\mapsto a_s$ of $G$ to $A$ which is continuous and such that
$a_{st}=a_s a_t$ for all $s,t\in G$.  The set of all 1-cocycles is denoted by
$Z^1(G, \mathbf{A})$.  Two cocycles $a$ and $a'$ are said to be {\it cohomologous}, 
 if there exists $b\in\mathbf{A}$ such that $a_s'=b^{-1}a_s b$. The quotient  of 
 $Z^1(G, \mathbf{A})$ by this equivalence relation is called the first 
 {\it cohomology set}  and is denoted by $H^1(G,\mathbf{A})$. 
The class of the unit cocycle is a distinguished element $\mathbf{1}$ in
the $H^1(G,\mathbf{A})$.  Notice that in general  there is no composition law on the 
set $H^1(G,\mathbf{A})$.
If $\mathbf{A}$ is an abelian group then the cohomology set $H^1(G,\mathbf{A})$
is an abelian group.

If $G$ is a profinite group, then 
\begin{equation}\label{eq2.2}
H^1(G,\mathbf{A})=\varinjlim H^1(G/U, \mathbf{A}^U),
\end{equation}
where $U$ runs through the set of open normal subgroups of $G$ and 
$\mathbf{A}^U$ is a subset of $\mathbf{A}$ fixed under action of $U$. 
The maps $H^1(G/U, \mathbf{A}^U)\to H^1(G,\mathbf{A})$ are injective. 

Let $k$ be a number field and $\bar k$ the algebraic closure of $k$. Denote by 
$Gal~(\bar k|k)$ the profinite Galois group of $\bar k$.  Let $V(k)$ be a projective 
variety over $k$ and $Aut ~V(k)$  the group of the $\bar k$-automorphisms
of $V(k)$. 
\begin{lemma}\label{lm2.1}
{\bf  [Serre 1997]  \cite[p. 124]{S}}
There exits a bijective correspondence between the twists of $V(k)$ and the set
$H^1(Gal~(\bar k|k), ~Aut ~V(k))$. 
\end{lemma}

\section{Proofs}
\subsection{Proof of theorem \ref{thm1.1}}
We shall split the proof in a series of lemmas.
\begin{lemma}\label{lm3.1}
The $H^1(Gal(\bar k|k), Aut_{\mathbf{C}}^{~ab}(V))$ is a stationary dimension group. 
\end{lemma}
\begin{proof}
It is known that the Galois group $Gal(\bar k|k)$ of the field extension $k\subset\mathbf{C}$ 
is a profinite group. We denote by $U_j$ an infinite ascending sequence of 
open normal subgroups of  $Gal(\bar k|k)$. In other words,
\begin{equation}\label{eq3.1}
Gal(\bar k|k)=\varinjlim U_j. 
\end{equation}
The corresponding Galois cohomology  (\ref{eq2.2}) can be written in the form
\begin{equation}\label{eq3.2}
H^1(Gal(\bar k|k), Aut_{\mathbf{C}}^{~ab}(V))=\varinjlim 
H^1\left(Gal(\bar k|k)/U_j,  \left(Aut_{\mathbf{C}}^{~ab}(V)\right)^{U_j}\right). 
\end{equation}
We shall write $\alpha_j$ to denote the injective map
\begin{equation}\label{eq3.3}
H^1\left(Gal(\bar k|k)/U_j,  \left(Aut_{\mathbf{C}}^{~ab}(V)\right)^{U_j}\right)
\to H^1(Gal(\bar k|k), Aut_{\mathbf{C}}^{~ab}(V)). 
\end{equation}
\begin{remark}
All the Galois cohomology in formulas (\ref{eq3.2}) and (\ref{eq3.3}) are abelian groups $\mathbf{Z}^{n_j}$, 
because the coefficient group  $Aut_{\mathbf{C}}^{~ab}(V)$   is taken to be abelian. 
\end{remark}

\bigskip
Recall that a {\it dimension group}  is an ordered abelian group which is the inductive 
limit of the sequence of abelian groups
\begin{equation}\label{eq3.4}
\mathbf{Z}^{n_1}\buildrel\rm\varphi_1\over\longrightarrow 
\mathbf{Z}^{n_2}\buildrel\rm\varphi_2\over\longrightarrow
\mathbf{Z}^{n_3}\buildrel\rm\varphi_3\over\longrightarrow   
\dots
\end{equation}
for some positive integers $n_j$ and some positive group homomorphisms 
$\varphi_j$, where $\mathbf{Z}^n$ is given the usual ordering 
\begin{equation}\label{eq3.5}
\left(\mathbf{Z}^n\right)^+=
\left\{(x_1,x_2,\dots,x_n)\in\mathbf{Z}^n ~:~ x_j\ge 0\right\}.
\end{equation}
The dimension group is called {\it stationary}, if $n_j=n=Const$ and 
$\varphi_j=\varphi=Const$.

\begin{figure}
\begin{picture}(300,110)(-90,20)
\put(40,70){\vector(1,-1){35}}
\put(90,70){\vector(0,-1){35}}
\put(45,83){\vector(1,0){30}}
\put(-110,80){$H^1\left(Gal(\bar k|k)/U_j,  \left(Aut_{\mathbf{C}}^{~ab}(V)\right)^{U_j}\right)$}
\put(80,80){$H^1\left(Gal(\bar k|k)/U_{j+1},  \left(Aut_{\mathbf{C}}^{~ab}(V)\right)^{U_{j+1}}\right)$}
\put(40,20){$H^1(Gal(\bar k|k), Aut_{\mathbf{C}}^{~ab}(V))$}
\put(60,90){$\varphi_j$}
\put(30,50){$\alpha_j$}
\put(100,50){$\alpha_{j+1}$}
\end{picture}
\caption{The group homomorphism $\varphi_j$.}
\end{figure}

\bigskip
We shall define a group homomorphism $\varphi_j: 
H^1\left(Gal(\bar k|k)/U_j,  (Aut_{\mathbf{C}}^{~ab}(V))^{U_j}\right)\to
H^1\left(Gal(\bar k|k)/U_{j+1},  (Aut_{\mathbf{C}}^{~ab}(V))^{U_{j+1}}\right)$
from a commutative diagram in Figure 1,  where the injective homomorphisms $\alpha_j$ are defined 
by the formula  (\ref{eq3.3}). Note that $H^1\left(Gal(\bar k|k)/U_j,  (Aut_{\mathbf{C}}^{~ab}(V))^{U_j}\right)
\cong\mathbf{Z}^{n_j}$ for an  integer $n_j \ge 0$;  the order on $\mathbf{Z}^{n_j}$ is defined by (\ref{eq3.5}). 
Moreover,  it is easy to see that $\varphi_j$ preserves the the order, i.e. $\varphi_j$ is a positive homomorphism.
Comparing formulas (\ref{eq3.2}) and (\ref{eq3.4}), 
we conclude that the cohomology group   $H^1(Gal(\bar k|k), Aut_{\mathbf{C}}^{~ab}(V))$ is a dimension group.

Let us show that the dimension group $H^1(Gal(\bar k|k), Aut_{\mathbf{C}}^{~ab}(V))$ is a stationary
dimension group, i.e. $n_j=n=Const$ and  $\varphi_j=\varphi=Const$. 
Indeed, notice that the shift $j\mapsto j+1$ in the RHS of formula (\ref{eq3.1}) corresponds to an automorphism 
of the group $Gal ~(\bar k|k)$.  
(Namely, the inductive limits $\varinjlim U_j$ and $\varinjlim U_{j+1}$ generate isomorphic profinite groups 
$Gal ~(\bar k|k)$.)  Such an automorphism gives rise to an automorphism of the Galois cohomology
(\ref{eq3.2}) and the corresponding dimension group (\ref{eq3.4}).  
It is not hard to see,  that such a dimension group is stationary and is generated by the above automorphism,
see [Blackadar 1986]  \cite[Theorem 7.3.2]{B}. 
\end{proof}

\begin{lemma}\label{lm3.3}
$H^1(Gal(\bar k|k), Aut_{\mathbf{C}}^{~ab}(V))\cong (K_0(\mathscr{A}_V), K_0^+(\mathscr{A}_V)).$
\end{lemma}
\begin{proof}
Recall that there exists an embedding $\mathscr{A}_V\hookrightarrow\mathscr{B}$,
where $\mathscr{B}$ is a stationary AF-algebra  such that 
$(K_0(\mathscr{A}_V), K_0^+(\mathscr{A}_V))\cong (K_0(\mathscr{B}), K_0^+(\mathscr{B}))$
\cite[Lemma 3.1]{Nik1}.   
Roughly speaking, such an embedding follows from the well known Pimsner Theorem
on the crossed product subalgebras of the AF-algebras [Pimsner 1983] \cite{Pim1}.

On the other hand, it is known that the Galois cohomology 
\linebreak
$H^1(Gal(\bar k|k), Aut_{\mathbf{C}}^{~ab}(V))$
is a functor from the category of projective varieties $V(k)$ to a category of abelian groups 
[Serre 1997]  \cite{S}. 
It follows from Elliott's classification theorem [Blackadar 1986] \cite[Corollary 7.4.2]{B}
that there exists  an AF-algebra $\mathscr{B}'$   such that 
\begin{equation}\label{eq3.6}
(K_0(\mathscr{B}'), K_0^+(\mathscr{B}'))\cong H^1(Gal(\bar k|k), Aut_{\mathbf{C}}^{~ab}(V)). 
\end{equation}
It is known that the Galois cohomology  group $H^1(Gal(\bar k|k), Aut_{\mathbf{C}}^{~ab}(V))$
is a functor on the varieties $V(k)$ [Serre 1997] \cite{S}.  
In view of (\ref{eq3.6}) the same is true for the  AF-algebra $\mathscr{B}'$. 
But an AF-algebra which is a functor on $V(k)$  must be Morita equivalent to the
Serre $C^*$-algebra $\mathscr{A}_V$. 
In other words, 
\begin{equation}\label{eq3.7}
(K_0(\mathscr{B}'), K_0^+(\mathscr{B}'))\cong(K_0(\mathscr{A}_V), K_0^+(\mathscr{A}_V)). 
\end{equation}
The conclusion of lemma \ref{lm3.3} follows from the formulas   (\ref{eq3.6}) and (\ref{eq3.7}).
\end{proof}

\bigskip
Theorem \ref{thm1.1}  follows from lemma \ref{lm3.3}. 

\bigskip
\begin{remark}
It is well known, that the $H^1(Gal(\bar k|k), Aut_{\mathbf{C}}^{~ab}(V))$ is a torsion 
group, i.e. each element of the group has finite order, see e.g. [Serre 1997] \cite[Chapter I, \S 2.2, Corollary 3]{S}.  
The $n$-dimensional Minkowski question-mark function 
\linebreak
$?^n(x): \mathbf{R}^n/\mathbf{Z}^n\to \mathbf{R}^n/\mathbf{Z}^n$
defines a functor $F$ from a  category of the stationary  scaled dimension groups
 $(K_0(\mathcal{A}_V), K_0^+(\mathcal{A}_V), \Sigma(\mathcal{A}_V))$
  to such of the infinite subgroups of the  torsion abelian group 
 $\mathbf{Q}^n/\mathbf{Z}^n$,  such that $F$ maps order-isomorphic scaled dimension groups
 to the isomorphic infinite torsion abelian groups \cite[Lemma 3.1]{Nik2}. 
 Thus the order-isomorphism (\ref{eq3.6}) is defined correctly. 
 \end{remark}

\subsection{Proof of corollary \ref{cor1.2}}
\begin{lemma}\label{lm3.4}
There exits a bijective correspondence between
the elements of the abelian group  $H^1(Gal~(\bar k|k), ~Aut_{\mathbf{C}}^{ab} ~V(k))$
and a subset of the set of  twists of  the variety $V(k)$.
\end{lemma}
\begin{proof}
A restriction of the coefficient group $Aut_{\mathbf{C}}~V(k)$ of the Galois cohomology 
$H^1(Gal~(\bar k|k), ~Aut_{\mathbf{C}}~V(k))$ to its unique  maximal abelian subgroup 
$Aut_{\mathbf{C}}^{ab}~V(k))$  defines an inclusion of the sets
\begin{equation}\label{eq3.8}
H^1(Gal~(\bar k|k), ~Aut_{\mathbf{C}}^{ab}~V(k)) \subseteq
H^1(Gal~(\bar k|k), ~Aut_{\mathbf{C}}~V(k)).  
\end{equation}
In view of the Serre's Lemma \ref{lm2.1},  one gets from the inclusion (\ref{eq3.8}) a bijection 
between the elements of  the abelian group $H^1(Gal~(\bar k|k), ~Aut_{\mathbf{C}}^{ab}~V(k))$ 
and a subset of the  set of twists of the variety $V(k)$.   Lemma \ref{lm3.4} is proved.  
\end{proof}

\begin{corollary}\label{cor3.5}
There exits a bijective correspondence between
the elements of the dimension group  $(K_0(\mathscr{A}_V), K_0^+(\mathscr{A}_V))$
and a subset of the set of  twists of  the variety $V(k)$.
\end{corollary}
\begin{proof}
This corollary  is an implication of theorem \ref{thm1.1} and lemma \ref{lm3.4}. 
\end{proof}

\begin{lemma}\label{lm3.6}
There exits a bijective correspondence between
the set of all  scaled dimension groups
$(K_0(\mathscr{A}_V), K_0^+(\mathscr{A}_V), u)$
and a subset of the set of  twists of  the variety $V(k)$.
\end{lemma}
\begin{proof}
It is known, that each $u\in K_0^+(\mathscr{A}_V)$ can be taken for an order-unit
of the scaled simple dimension group $(K_0(\mathscr{A}_V), K_0^+(\mathscr{A}_V), u)$  [Blackadar 1986]  \cite[Section 6.2]{B}. 
Moreover, the obtained scaled dimension groups $(K_0(\mathscr{A}_V), K_0^+(\mathscr{A}_V), u)$
are distinct for different elements $u\in K_0^+(\mathscr{A}_V)$ 
and any such  group can be  obtained in this way 
[Blackadar 1986]  \cite[Section 6.2]{B}.
Thus lemma \ref{lm3.6} follows from the corollary \ref{cor3.5}. 
\end{proof}

\begin{lemma}\label{lm3.7}
There exits a bijective correspondence between
the set of all Morita equivalent but pairwise non-isomorphic  Serre $C^*$-algebras $\mathscr{A}_V$
and  $\mathscr{A}_V\otimes M_n(\mathbf{C})$  and a subset of the set of  twists of  the variety $V(k)$.
\end{lemma}
\begin{proof}
Recall that there exists an embedding $\mathscr{A}_V\hookrightarrow\mathscr{B}$, 
where $\mathscr{B}$ is an AF-algebra, such that
\begin{equation}\label{eq3.9}
(K_0(\mathscr{A}_V), K_0^+(\mathscr{A}_V))\cong (K_0(\mathscr{B}), K_0^+(\mathscr{B})). 
\end{equation}
It is known that the dimension group 
$(K_0(\mathscr{B}), K_0^+(\mathscr{B}))$  is an invariant of the Morita equivalence
of the AF-algebra $\mathscr{B}$, while the scaled dimension group 
\linebreak
$(K_0(\mathscr{B}), K_0^+(\mathscr{B}), u)$ is an invariant of the isomorphism of 
$\mathscr{B}$,  see e.g. [Blackadar 1986]  \cite[Theorem 7.3.2]{B}. 
In view of (\ref{eq3.9}), the same is true of the Serre $C^*$-algebra 
$\mathscr{A}_V$.  Lemma \ref{lm3.7} follows from the corollary \ref{cor3.5} and 
lemma \ref{lm3.6}. 
\end{proof}

\bigskip
Corollary \ref{cor1.2}  follows from lemma \ref{lm3.7} and the definition of a twist.

\section{Rational elliptic curves}
To illustrate corollary \ref{cor1.2},   we shall consider the case $V(k)\cong\mathscr{E}(k)$,
where $\mathscr{E}(k)$ is a rational elliptic curve. 
We briefly review  the related definition and facts. 

\subsection{Elliptic curves}
By an {\it elliptic curve} we understand  the subset of the complex projective plane of the form
\begin{equation}\label{eq4.1}
\mathscr{E}(k)=\{(x,y,z)\in \mathbf{C}P^2 ~|~ y^2z=x^3+Axz^2+Bz^3\},
\end{equation}
where $A,B\in k$  are some constants.  Recall that the number 
$j(\mathscr{E})=1728 (4A^3)/(4A^3+27B^2)$  is an invariant of the 
$\mathbf{C}$-isomorphisms of the elliptic curve $\mathscr{E}(k)$. 
The twists $\mathscr{E}_t(k)$ of $\mathscr{E}(k)$ are given by  the equations
\begin{equation}\label{eq4.2}
\begin{cases}
y^2z=x^3+t^2A xz^2+t^3B z^3,\qquad  \hbox{if}   ~j(\mathscr{E})\ne 0, 1728\cr
y^2z=x^3+tA xz^2, \qquad \qquad\qquad  \hbox{if}   ~j(\mathscr{E})=1728\cr
y^2z=x^3+tB z^3,  \quad\qquad\qquad\qquad  \hbox{if}   ~j(\mathscr{E})=0, 
\end{cases}
\end{equation}
where $t\in k$, see e.g.  [Silverman 1985] \cite[Proposition 5.4]{S1}. 
It is easy to verify, that $j(\mathscr{E}_t(k))=j(\mathscr{E}(k))$.

\subsection{Noncommutative tori}
A $C^*$-algebra $\mathscr{A}_{\theta}$ on two generators $u$ and $v$ satisfying the relation 
$vu=e^{2\pi i\theta}uv$ for a  constant $\theta\in\mathbf{R}$ is called the {\it noncommutative torus}. 
The well known Rieffel's Theorem \cite[Theorem 1.1.2]{N}  says that the algebra $\mathscr{A}_{\theta}$ is Morita equivalent to the algebra
$\mathscr{A}_{\theta'}$, if and only if,  
\begin{equation}\label{eq4.3}
\theta'=\frac{a\theta+b}{c\theta+d}  \quad\hbox{for a matrix } 
\left(
\begin{matrix}
a & b\cr
c & d
\end{matrix}
\right)\in SL(2, \mathbf{Z}).  
\end{equation}
In contrast, the algebra $\mathscr{A}_{\theta}$ is isomorphic  to the algebra
$\mathscr{A}_{\theta'}$ if and only if $2(\theta -\theta')\in\mathbf{Z}$. 
Notice that in terms of the continued fraction 
\begin{equation}\label{eq4.4}
\theta=a_0+{1\over\displaystyle a_1+
{\strut 1\over\displaystyle a_2\displaystyle +\dots}}:=
[a_0, a_1, a_2, \dots]
\end{equation}
attached to the parameter $\theta$, it means that the $\mathscr{A}_{\theta'}$
is Morita equivalent to the $\mathscr{A}_{\theta}$, if and only if, the continued 
fraction of $\theta'$ coincides with such of $\theta$ everywhere  but  a finite number
of terms.  In other words,  an infinite tail  of the corresponding continued fractions
must be the same. 
Clearly,  the $\mathscr{A}_{\theta'}$ is isomorphic to the $\mathscr{A}_{\theta}$, if and only if, the continued 
fraction of $\theta'$ coincides  with such of $\theta$.  
\begin{remark}\label{rmk4.1}
An infinite tail of continued fraction (\ref{eq4.4}) is an invariant of the Morita equivalence of 
the algebra  $\mathscr{A}_{\theta}$.  Such a tail is an analog of the $j$-invariant 
of an elliptic curve.
\end{remark}

\subsection{Twists of $\mathscr{E}(k)$}
It is known that the Serre $C^*$-algebra of an elliptic curve $\mathscr{E}(k)$ is isomorphic 
to the $\mathscr{A}_{\theta}$  \cite[Theorem 1.3.1]{N}. 
Moreover, if $k$ is a number field, then $\theta$ is a real quadratic number,
i.e. the irrational root of a quadratic polynomial with integer coefficients,
see  \cite[Theorem 1.4.1]{N}. 
For such a number, the continued fraction (\ref{eq4.4}) must be eventually periodic,
i.e. 
\begin{equation}\label{eq4.5}
\theta=[a_0,\dots, a_N; \overline{b_1,\dots, b_n}],
\end{equation}
where  $(b_1,\dots, b_n)$ is the minimal period of the fraction. 
\begin{corollary}\label{cor4.2}
The period $(b_1,\dots, b_n)$ of the continued fraction (\ref{eq4.5}) is an invariant of the twists
 (\ref{eq4.2}), while $(a_0,\dots,a_N)$ in (\ref{eq4.5}) depend on the twist 
 parameter $t\in k$ in (\ref{eq4.2}).  
\end{corollary}
\begin{proof}
Up to a cyclic permutation,   the  period $(b_1,\dots, b_n)$  is the Morita invariant 
of the algebra $\mathscr{A}_{\theta}$,  see  remark \ref{rmk4.1}.  The corollary 
\ref{cor4.2} follows from the corollary \ref{cor1.2}. 
\end{proof}


\bibliographystyle{amsplain}


\end{document}